\documentclass[12pt]{amsart}
\usepackage{amsmath}
\usepackage{amssymb}
\usepackage{amsfonts}
\usepackage{mathrsfs}
\usepackage{graphicx}
\usepackage{epsfig}
\usepackage[T1]{fontenc}
\numberwithin{equation}{section}       

\theoremstyle{plain}
\newtheorem{theorem}{Theorem}[section]

\newtheorem{prop}{Proposition}[section]

\newtheorem{coro}[prop]{Corollary}
\newtheorem{lemma}[prop]{Lemma}

\theoremstyle{definition}

\theoremstyle{remark}
\newtheorem{remark}[prop]{Remark}

\newtheoremstyle{citing}
  {3pt}
  {3pt}
  {\itshape}
  {}
  {\bfseries}
  {.}
  {.5em}
  {\thmnote{#3}}

\theoremstyle{citing}

\DeclareMathAlphabet{\mathpzc}{OT1}{pzc}{m}{it} 

\newcommand{\C}{\mathbb{C}}

\newcommand{\N}{\mathbb{N}}

\newcommand{\Z}{\mathbb{Z}}

\newcommand{\teta}{\widetilde{\teta}}

\newcommand{\dist}{d}

\DeclareMathOperator{\id}{Id}

\begin{document}

\title[]{No hyperbolic sets in $J_\infty$ for infinitely renormalizable quadratic polynomials}

\author{Genadi Levin}

\address{Institute of Mathematics, The Hebrew University of Jerusalem, Givat Ram,
Jerusalem, 91904, Israel}

\email{levin@math.huji.ac.il}

\author{Feliks Przytycki}

\address{Institute of Mathematics, Polish Academy of Sciences, \'Sniadeckich St., 8, 00-956 Warsaw,
Poland}

\email{feliksp@impan.pl}

\date{\today}

\maketitle

\begin{abstract}
Let $f$ be an infinitely renormalizable quadratic polynomial and $J_\infty$ be the intersection of forward orbits of "small" Julia sets
of its simple renormalizations. We prove that $J_\infty$ contains no hyperbolic sets.
\end{abstract}

\section{Introduction}
Let $f$ be a rational function of degree at least $2$ considered as a dynamical systems $f:\hat\C\to\hat\C$ on the Riemann sphere $\hat\C$. An $f$-invariant compact set $X\subset \hat\C$ is said to be {\it hyperbolic}
if $f: X\to X$ is uniformly expanding
, i.e., for some $C>0$ and $\lambda>1$, $|D(f^m)(x)|\ge C\lambda^m$
for all $x\in X$ and all $m\ge 0$
(here $D$ stands for the spherical derivative and $f^m$ is $m$-iterate of $f$). In particular, any repelling periodic orbit of $f$ is a hyperbolic set.
The closure of all repelling periodic orbits of $f$ is the Julia set $J(f)$ of $f$.
Hyperbolic sets of $f$ are contained in $J(f)$. Apart of repelling periodic orbits, $f$ admits plenty of infinite (Cantor) hyperbolic sets \cite{PU}.
Attracting periodic orbits (if any) along with their basins are contained in the complement $\hat\C\setminus J(f)$ (which is called the Fatou set of $f$). See e.g. \cite{CG} for an introduction to complex dynamics and \cite{Mary} for a recent survey.

If $J(f)$ is a hyperbolic set by itself, i.e., $f:J(f)\to J(f)$ is uniformly expanding, then $f$ is called a hyperbolic rational map.
Equivalently, all critical points of $f$ are in basins of attracting cycles.
Hyperbolic rational maps are analogous to Axiom A diffeomorphisms and their dynamics has been intensively studied and very well understood.
The famous 'Density of Hyperbolicity Conjecture (DHC)' in holomorphic dynamics - sometimes also called the Fatou conjecture - asserts that any rational map (polynomial) can be approximated by hyperbolic rational maps
(polynomials) of the same degree.

In what follows $f$ (unless mentioned explicitly) is a quadratic polynomial $f_c(z)=z^2+c$. The DHC (as well as a more general MLC: Mandelbrot set Locally Connected) is widely open for the
quadratic family $f_c$, too (DHC for $f_c$ as strongly believed accumulates in itself the essence of the general DHC). After a breakthrough work of Yoccoz \cite{H} on the MLC, the only obstacle for proving DHC for quadratic polynomials are so-called infinitely renormalizable ones, see \cite{mcm}.

Somewhat informally, a quadratic polynomial $f_c$ with connected Julia set is called renormalizable if, for some topological disks $U,V$  around the critical point $0$ of $f_c$
and for some $p\ge 2$ (called period of the renormalization), the restriction $f_c^p:U\to V$
is conjugate to another quadratic polynomial $f_{c'}$ with connected Julia set (see \cite{DH} for exact definitions and the theory of polynomial-like mappings).
The map $F:=f_c^p:U\to V$ is then a {\it renormalization} of $f_c$
and the set $K(F)=\{z\in U: F^{n}(z)\in U \mbox{ for all } n\ge 1\}$ is a {\it "small" (filled in ) Julia set} of $f_c$.
If $f_{c'}$ is renormalizable by itself, then $f_c$ is called twice renormalizable, etc.
If $f_c$ admits infinitely many renormalizations, it is called {\it infinitely renormalizable}.
Recall that the renormalization $F$ is {\it simple} if any two sets $f^i(K(f))$, $f^j(K(F))$, $0\le i<j\le p-1$,
are either disjoint or intersect each other at a unique point which does not separate either of them.

To state our main result - which is Theorems \ref{thm:exp} - let $f(z)=z^2+c$ be infinitely renormalizable.
Let $1=p_0<p_1<...<p_n<...$ be the sequence
of consecutive periods of simple renormalizations of $f$ and $J_n$ denotes the "small" Julia set of the
$n$-renormalization (where $J_0=J(f)$). Then $p_{n+1}/p_n$ is an integer, $f^{p_n}(J_n)=J_n$, for any $n$, and $\{J_n\}_{n=1}^\infty$ is a strictly decreasing sequence of continua without interior, all containing $0$.
Let
$$J_\infty=\cap_{n\ge 0}\cup_{j=0}^{p_n-1}f^j(J_n)$$
be the intersection of orbits of the "small" Julia sets.
$J_\infty$ is a compact $f$-invariant set which contains the omega-limit set $\omega(0)$ of $0$.
Each component of $J_\infty$ is wandering, in particular, $J_\infty$ contains no periodic orbits of $f$.
Note that a hyperbolic set in $J_\infty$ (if existed) could not be repelling,
that is any forward orbit of a point sufficiently close to this set must be in the set itself, since otherwise shadowing periodic orbits must be
in $J_\infty$.

It is shown in \cite{LPS} that the low Lyapunov exponent of the critical value $c\in J(f_c)$ is always non-negative.
In the considered case, $c\in J_\infty$. We prove:
\begin{theorem}\label{thm:exp}
$J_\infty$ contains no hyperbolic sets.
\end{theorem}
Combined with the Fatou-Mane theorem \cite{mane} Theorem \ref{thm:exp} immediately implies
\begin{coro}
$\omega(x)\cap \omega(0)\neq\emptyset$, for the omega-limit set $\omega(x)$ of every $x\in J_\infty$.
\end{coro}
The conclusion of Theorem \ref{thm:exp} would obviously hold provided
\begin{equation}\label{point}
J_\infty \mbox{ is totally disconnected. }
\end{equation}
(\ref{point}) is true indeed for many classes of maps (including real ones) where it follows from 'complex bounds' \cite{Su} (meaning roughly that the sequence of renormalizations is compact)
\cite{ls}, \cite{grsw}, \cite{lyya}, \cite{kahn}, \cite{kahnl1}, \cite{kahnl2}.
See also \cite{kss1}, \cite{kss2}.
However, (\ref{point}) breaks down in general:
see \cite{Mi0}, \cite{So} for the existence of such maps
and \cite{L}, \cite{L1},  \cite{L1add} (see also \cite{CS}) for explicit combinatorial conditions on $f_c$ for (\ref{point}) to fail.
Yoccoz \cite{Yfields} posed a problem to find a necessary and sufficient condition on the combinatorics of $f_c$ for
(\ref{point}) to hold. At present, the gap between known sufficient and necessary conditions is still very big.

Another well-known open problem is to give necessary and sufficient conditions so that the Julia set $J(f)$ is locally-connected. For example, if (\ref{point}) does not hold then $J(f)$ is not locally-connected.
Theorem\ref{thm:exp} implies
\begin{theorem}\label{thm:explc}
Let $f(z)=z^2+c$ and $f$ has no irrational indifferent periodic orbits.
Then $J(f)$ is locally-connected at every point of any hyperbolic set $X$ of $f$.
In particular, there are at least one and at most finitely many external rays landing at each $x\in X$.
\end{theorem}
\begin{remark}
The case that $f$ does have an irrational cycle seems to be open and requires a separate consideration, see \cite{Che} though.
Note also that Theorem \ref{thm:explc} removes the only restriction in Proposition 2.11 of \cite{BLy} for degree $2$ polynomials without irrational cycles.
\end{remark}
Theorem \ref{thm:explc} has been known for the following quadratic maps $f$.
If $f$ has an attracting cycle, then $f$ is hyperbolic and the whole $J(f)$ is locally-connected.
The same conclusion holds if $f$ has a parabolic cycle~\cite{DH}.
The first part of Yoccoz's result (see e.g.,~\cite{H}) says that $J(f)$ is locally-connected if
$f$ has no indifferent irrational cycles and at most finitely many times renormalizable.
This allows us to reduce the proof of Theorem \ref{thm:explc} to the case of $f$ as in Theorem \ref{thm:exp}, hence, by the latter,
to the case when $X$ is disjoint from $J_\infty$ in which case it is well-known that Yoccoz puzzle pieces shrink
to each point of $X$ \cite{Mi0}, \cite{lfund}. This shows that $J(f)$ is locally connected at points of $X$. The last claim follows then from \cite{kiwi}, see also \cite{th} and \cite{lfund}.

{\bf Acknowledgment.} We thank Weixiao Shen for helpful comments. We thank the referee for many useful remarks that helped to improve the paper.


\section{Preliminaries}\label{prel}
Here
we collect, for further references, necessary notations and general facts
which are either well-known \cite{mcm}, \cite{Mi1} or follow readily from the known ones.
Let $f(z)=z^2+c$ be infinitely renormalizable. We keep the notations of the Introduction.

{\bf (A)}. Let $G$ be the Green function of the basin of infinity $A(\infty)=\{z| f^n(z)\to \infty, n\to\infty\}$ of $f$ with the standard normalization
at infinity $G(z)=ln|z|+O(1/|z|)$. The external ray $R_t$ of argument $t\in {\bf S^1}={\bf R}/{\bf Z}$ is a gradient line to the level sets of $G$
that has the (asymptotic) argument $t$ at $\infty$.
$G(z)$ is called the (Green) level of $z\in A(\infty)$ and the unique $t$ such that $z\in R_t$
is called the (external) argument (or angle) of $z$.
A point $z\in J(f)$ is accessible if there is an external ray $R_t$ which lands at (i.e., converges to) $z$. Then $t$ is called an (external) argument (angle) of $z$.

Let $\sigma: {\bf S^1}\to {\bf S^1}$ be
the doubling map $\sigma(t)=2t(mod 1)$. Then $f(R_t)=R_{\sigma(t)}$.

{\bf (B)}.
Given a small Julia set $J_n$ containing $0$, sets $f^j(J_n)$ ($0\le j< p_n$) are called small Julia sets of level $n$. Each $f^j(J_n)$ contains $p_{n+1}/p_n\ge 2$ small Julia sets
$f^{j+k p_n}(J_{n+1})$, $0\le k<p_{n+1}/p_n$, of level $n+1$.
We have $J_n=-J_n$. Since all renormalizations are simple, for $j\neq 0$, the symmetric companion $-f^j(J_n)$ of $f^j(J_n)$ can intersect the orbit $orb(J_n)=\cup_{j=0}^{p_n-1}f^j(J_n)$ of $J_n$
only at a single point which is preperiodic.
On the other hand,
since only finitely many external rays converge to each periodic point of $f$, the set $J_\infty$ contains no periodic points. In particular,
each component $K$ of $J_\infty$ is wandering, i.e., $f^i(K)\cap f^j(K)=\emptyset$ for all $0\le i<j<\infty$.
All this implies that $\{x,-x\}\subset J_\infty$ if and only if $x\in K_0:=\cap_{n=1}^\infty J_n$.

Given $x\in J_\infty$, for every $n$, let $j_n(x)$ be the unique $j\in\{0,1,\cdots, p_n-1\}$ such that
$x\in f^{j(x)}(J_n)$. Let $J_{x,n}=f^{j_n(x)}(J_n)$ be a small Julia set of level $n$ containing $x$ and
$K_x=\cap_{n\ge 0}J_{x, n}$, a component of $J_\infty$ containing $x$.

In particular, $K_0=\cap_{n\ge 0}J_n$ is the component of $J_\infty$ containing $0$ and $K_c=\cap_{n=1}^\infty f(J_n)$, the component containing $c$.

The map $f:K_x\to K_{f(x)}$ is one-to-one if $x\notin K_0$ while $f:K_0\setminus \{0\}\to K_{c}\setminus \{c\}$ is two-to-one.
Moreover, for every $y\in J_\infty$, $f^{-1}(y)\cap J_\infty$ consists of two points if $y\in K_c\setminus \{c\}$ and consists of a single point
otherwise.

{\bf (C)}. Given $n\ge 0$, the map $f^{p_n}:f(J_n)\to f(J_n)$ has two fixed points:
the separating fixed point $\alpha_n$ (that is, $f(J_n)\setminus \{\alpha_n\}$ has at least two components) and
the non-separating $\beta_n$ (so that $f(J_n)\setminus\{\beta_n\}$ has a single
component).

For every $n>0$, there are two rays $R_{t_n}$ and $R_{\tilde t_n}$ ($0<t_n<\tilde t_n<1$)
to the non-separating fixed point $\beta_n\in f(J_n)$
of $f^{p_n}$ such that the component $\Omega_n$ of ${\bf C}\setminus (R_{t_n}\cup R_{\tilde t_n}\cup \beta_n)$
which does not contain $0$ has two characteristic properties:

(i) $\Omega_n$ contains $c$ and contains no the forward orbit
of $\beta_n$,

(ii) for every $1\le j\le p_n$, consider arguments (angles) of the the external rays which land at
$f^{j-1}(\beta_n)$.
The angles split ${\bf S^1}$ into finitely many arcs. Then the arc
$$S_{n,1}=[t_n, \tilde t_n]=\{t: R_t\subset \Omega_n\}$$
has the smallest length among all these arcs.

Denote
$$t_n'=t_n+\frac{\tilde t_n-t_n}{2^{p_n}}, \ \ \tilde t_n'=\tilde t_n-\frac{\tilde t_n-t_n}{2^{p_n}}.$$
The rays $R_{t_n'}$, $R_{\tilde t_n'}$ land at a common point $\beta_n'\in f^{-p_n}(\beta_n)\cap \Omega_n$.
Introduce an (unbounded) domain $U_n$ with the boundary to be two curves
$R_{t_n}\cup R_{\tilde t_n}\cup \beta_n$ and $R_{t_n'}\cup R_{\tilde t_n'}\cup \beta_n'$.
In other words, $U_n$ is a component of $f^{-p_n}(\Omega_n)$ which is contained in $\Omega_n$.
Then $c\in U_n$ and $f^{p_n}: U_n\to \Omega_n$ is a two-to-one branched covering so that
\begin{equation}\label{j1}
f(J_n)=\{z| f^{kp_n}(z)\in \overline U_n, G(f^{kp_n}(z))\le 10, k=0,1,...\}.
\end{equation}
Moreover, for any $n$, the closure
of $U_{n+1}$ is contained in
$U_n$.
We denote
$$s_{n,1}=[t_n, t_n']\cup [\tilde t_n', \tilde t_n].$$
Then $s_{n,1}\subset S_{n,1}$ and
$$\sigma^{p_n}: s_{n,1}\to S_{n,1}$$
so that $\sigma^{p_n}$ is a homeomorphism of each component of $s_{n,1}$ onto $S_{n,1}$. End points $t_n, \tilde t_n$ of $S_{n,1}$ are fixed points of $\sigma^{p_n}$.
It's important to note that $S_{n+1,1}\subset S_{n,1}$, $s_{n,1}\subset s_{n+1,1}$ for all $n$ and the length $(\tilde t_n-t_n)/2^{p_n}$ of each
of the two components of
$s_{n,1}$ tends to zero as $n\to\infty$ (while the length $|t_n-\tilde t_n|$ of $S_{n,1}$ can stay away from zero).

From now on, {\it given a compact set $Y\subset J(f)$ denote by $\tilde Y$
the set of arguments of the external rays which have their limit sets in $Y$.}

For each $k\ge 0$ the boundary of the set $\{z: f^{kp_n}(z)\in \overline U_{n,}\}$
consists of rays and (pre)periodic points.
It follows that if a ray has at least one limit point in $f(J_n)$ then all its limit points are in $f(J_n)$, and (\ref{j1}) implies that
\begin{equation}\label{tildej1}
\widetilde{f(J_n)}=\{t| \sigma^{j p_n}(t)\in s_{n,1}, j=0,1,...\}.
\end{equation}
So $\widetilde{f(J_n)}$ is a Cantor set, in particular, closed.
Let us show that
\begin{equation}\label{tildekc}
\tilde K_c=\cap_{n=1}^\infty s_{n,1}.
\end{equation}
Indeed, $t\in\tilde K_c$ implies $t\in\widetilde{f(J_n)}\subset s_{n,1}$, for each $n$.
Vice versa, let $t\in\cap_{n=1}^\infty s_{n,1}$. It is enough to show that $t\in\widetilde{f(J_n)}$ for each $n$
(which would indeed imply that the limit set of the ray $R_t$ belongs to $f(J_n)$ for all $n$, i.e., $t\in\tilde K_c$).
Fix $n$ and find a sequence $t_m\in\partial s_{m,1}$,
such that $t_m\to t$ as $m\to\infty$. On the other hand, $\partial s_{m,1}\subset\widetilde{f(J_m)}\subset\widetilde{f(J_n)}$
because $f(J_m)\subset f(J_n)$ for $m>n$. As $\widetilde{f(J_n)}$ is closed, $t\in \widetilde{f(J_n)}$. This proves (\ref{tildekc}).

It implies that $\tilde K_c$ is either a single-point set or
a two-point set. In particular, $K_c$ contains at most two different accessible points.
As $f: K_0\setminus \{0\}\to K_c\setminus \{c\}$ is two-to-one, $\tilde K_0=\sigma^{-1}(\tilde K_c)$ consists of either 2 or 4 points.

Let us give,
for completeness of the picture, a similar description of
$\tilde K$ for each
component $K$ of $J_\infty$ (Lemma \ref{tilde}, see below).
It will be needed for the proof of Lemma \ref{2p}, part (ii). Note, however, that part (ii) of Lemma \ref{2p} is not used in the proof of the main result.

Let
$$s_{n,j}=\sigma^{j-1}(s_{n,1})=[t_{n,j}, t_{n,j}']\cup [\tilde t_{n,j}', \tilde t_{n,j}], 1\le j\le p_n,$$
where $t_{n,j}=\sigma^{j-1}(t_n)$, $\tilde t_{n,j}=\sigma^{j-1}(\tilde t_n)$, $t_{n,j}'=\sigma^{j-1}(t_n')$, and $\tilde t_{n,j}'=\sigma^{j-1}(\tilde t_n')$.
Then
\begin{equation}\label{wind}
t_{n,j}'-t_{n,j}=\tilde t_{n,j}-\tilde t_{n,j}'=\frac{\tilde t_n-t_n}{2^{p_n-j+1}}<\tilde t_n-t_n<1/2.
\end{equation}
So $\sigma^{j-1}: s_{n,1}\to s_{n,j}$ is a homeomorphism and $s_{n,j}$ has two components ('windows') $[t_{n,j}, t_{n,j}']$ and $[\tilde t_{n,j}', \tilde t_{n,j}]$ of equal length.
However, $\sigma^j(S_{n,1})$ can cover the whole circle ${\bf S^1}$ for some $j<p_n$.
For this reason, an analogue of (\ref{j1}) breaks down if $j$ is big.
For $j=1,2,...,p_n$, let $U_{n,j}=f^{j-1}(U_n)$ and $\beta_{n,j}=f^{j-1}(\beta_n)$. The domain $U_{n,j}$ is bounded by
two rays $R_{t_{n,j}}\cup R_{\tilde t_{n,j}}$ converging to $\beta_{n,j}$ and completed by $\beta_{n,j}$ along with two rays
$R_{t_{n,j}'}\cup R_{\tilde t_{n,j}'}$ completed by their common limit point $f^{j-1}(\beta_n')$.
Let $U^1_{n,j}$ be a component of $f^{-(p_n-j+1)}(U_n)$ which is contained in $U_{n,j}$.
Then
\begin{equation}\label{mapU}
f^{p_n}: U^1_{n,j}\to U_{n,j}
\end{equation}
is a two-to-one branched covering,
and
\begin{equation}\label{jany}
f^{j-1}(J_n)=\{z| f^{kp_n}(z)\in \overline U^1_{n, j}, G(f^{kp_n}(z))\le 10, k=0,1,...\}.
\end{equation}
Note that this is consistent with (\ref{j1}) for $j=1$.
Similar to $j=1$, if a ray has at least one limit point in $f^{j-1}(J_n)$ then all its limit points are there.

Let $s^1_{n,j}$ be the set of arguments of rays entering $\overline U^1_{n,j}$. Then $s^1_{n,j}\subset s_{n,j}$ consists of two pairs of components
($1$-'windows') where each
pair is adjacent to two end points of one of the 'windows' of $s_{n,j}$. Moreover, $\sigma^{p_n}: s^1_{n,j}\to s_{n,j}$ is a two-to-one covering
which maps each of four $1$-'windows' of $s^1_{n,j}$ homeomorphically onto one of the two 'windows' of $s_{n,j}$.
Correspondingly, for each $j=1,2,...,p_n$,
\begin{equation}\label{tildejany}
\widetilde{f^j(J_n)}=\{t|\sigma^{k p_n}(t)\in s^1_{n,j}, k=0,1,...\}.
\end{equation}
As (\ref{jany}) is consistent with (\ref{j1}) for $j=1$, (\ref{tildejany}) is consistence with (\ref{tildej1}) for $j=1$.

Given $n,j$, denote by $\Delta_{n,j}$ the length of each 'window' of $s_{n,j}$ and by
$\Delta^1_{n,j}$ the length of each 'window' of $s^1_{n,j}$. By (\ref{wind}), $\Delta_{n,j}=\frac{\tilde t_n-t_n}{2^{p_n-j+1}}$
and $\Delta^1_{n,j}=\frac{\Delta_{n,j}}{2^{p_n}}$.
So $\Delta^1_{n,j}<2^{-p_n}\to 0$ uniformly in $j$ as $n\to\infty$.

Let $K$ be a component of $J_\infty$:
$$K=\cap_{n=1}^\infty f^{j_n}(J_n),$$
where $1\le j_n\le p_n$ and $f^{j_{n+1}}(J_{n+1})\subset f^{j_n}(J_n)$.
\begin{lemma}\label{tilde}
\begin{equation}\label{tildek}
\tilde K=\cap_{n=1}^\infty s^1_{n,j_n}
\end{equation}
where $\{s^1_{n,j_n}\}_{n=1}^\infty$ is a decreasing sequence of compacts each consisting of
four $1$-'windows' with equal lengths tending to zero as $n\to\infty$. In particular,
$\tilde K$ consists of at most $4$ points.
There is an alternative (1)-(2):
\begin{enumerate}
\item either $p_n-j_n\to\infty$ as $n\to\infty$ so that the length of each 'window' of $s_{n,j_n}$ tends to zero,
moreover,
\begin{equation}\label{tildek1}
\tilde K=\cap_{n=1}^\infty s_{n,j_n}
\end{equation}
so that $\# \tilde K\in\{1,2\}$ in this case.
\item there is $N\ge 0$ such that $p_n-j_n=N$ for all $n$, so that $f^N(K)=K_0$.
\end{enumerate}

\end{lemma}
\begin{proof} (\ref{tildek}) is very similar to the proof of (\ref{tildekc}) and is left to the reader.
As for the alternative (1)-(2), assume that there is $N\ge 0$ such that, for an infinite subsequence $(n')\subset\N$, $p_{n'}-j_{n'}=N$. Then
$f^N(K)=\cap_{n'} f^{n'}(J_{n'})=\cap_{n} f^{n}(J_n)=K_0$, hence, $p_n-j_n=N$ for all $n$. This explains why
either $p_n-j_n\to\infty$ or $p_n-j_n=N$ for some $N\ge 0$ and all $n$. Consider the case $p_n-j_n\to\infty$.
Then by (\ref{wind}) the length $\Delta_{n,j_n}$ of each 'window' of $s_{n,j_n}$ tends to zero uniformly in $j_n$.
Repeating again the proof of (\ref{tildekc}) we get that
$\cap_{n=1}^\infty s_{n,j_n}=\cap_{n=1}^\infty s'_{n,j_n}$.
This settles the alternative.
\end{proof}


{\bf (D1)}. Consider in more detail the case when  $\tilde K_c=\{\tau_1, \tau_2\}$, $\tau_1\neq\tau_2$.
Let $S_c$ be the shortest arc in ${\bf S^1}$ with the end points $\tau_1$, $\tau_2$. It follows from (C):

$1_\infty$: $\sigma^k(\tau_i)\notin S_c$, for $i=1,2$ and all $k>0$,

$2_\infty$: for each $k>0$, the length of arcs with the end points $\sigma^k(\tau_1)$, $\sigma^k(\tau_2)$
is bigger than or equal to the length of $S_c$,

$3_\infty$: (unlinking) for each positive $j\not=k$, one of the two arcs
${\bf S^1}\setminus \{\sigma^k(\tau_1), \sigma^k(\tau_2)\}$
contains both points $\sigma^j(\tau_1), \sigma^j(\tau_2)$.
Furthermore, as $\tilde K_0=\sigma^{-1}(\{\tau_1,\tau_2\})$, the set $\tilde K_0$ splits the unit circle ${\bf S^1}$
into $4$ arcs in such a way that, for each $j\ge 0$, both points $\sigma^j(\tau_1), \sigma^j(\tau_2)$ have to lie in one and only one of these arcs.
In particular, $\sigma^j(\tau_1), \sigma^j(\tau_2)$ lie in one and only one of semicircles
${\bf S^1}\setminus\sigma^{-1}(\tau_i)$, for each $i=1,2$.
\begin{remark}\label{lamin}
Let us put the 'unlinking' property in the context of Thurston's 'laminations' \cite{th}. Consider a topological model
of $J(f)$ by shrinking all components of $J_\infty$ as well as all their preimages to points.
More formally, let's build a lamination (a special equivalent relation on ${\bf S^1}$) as follows \cite{th}, \cite{kiwi}:
as all periodic points of $f$ are repelling, each (pre)periodic point is a landing point of at least one and at most finitely many rays.
Let us identify arguments of such rays for each such point and take a closure of this partial relation to the whole ${\bf S^1}$.
The resulting relation is invariant under the map $\sigma: {\bf S^1}\to {\bf S^1}$.
For visualization, this relation is usually extended to the closed unit disk by taking the closed convex hull in the Euclidean plane of each equivalence class.
Then obviously different convex hulls (classes) are disjoint.
For example, $\sigma^{-1}(\{\tau_1,\tau_2\})$ is a single class, and
$\sigma^j(\tau_1), \sigma^j(\tau_2)$ is the other one, for each $j\ge 0$, so their convex hulls are disjoint.
The property $\# \tilde K\le 2$ (which is proved in Lemma \ref{tilde}) for any component $K$ of $J_\infty$ other than preimages of $K_0$
is, in fact, a very particular case of the fundamental 'no wandering triangle' property of unicritical laminations
\cite{th}, \cite{lfund}.
\end{remark}

{\bf (D2)}. Given $\nu\in [0,1)$ there exists a unique {\it minimal} rotation set $\Lambda_{\nu}\subset {\bf S^1}$ of
$\sigma: {\bf S^1}\to {\bf S^1}$ with the rotation number $\nu$ \cite{BS}. Recall that a closed subset $\Lambda$ of ${\bf S^1}$ is a rotation set of $\sigma$ with the rotation number $\nu$ if $\sigma(\Lambda)\subset\Lambda$ and $\sigma :\Lambda\to\Lambda$ extends to a map of ${\bf S^1}$ which lifts to an orientation preserving non-decreasing
continuous map $F: {\bf R}\to {\bf R}$ with $F-\id$ to be $1$-periodic and the fractional part of the rotation number of $F$ to be equal to $\nu$.
Then \cite{BS} $\nu$ is irrational if and only if $\Lambda_\nu$ is infinite, in this case there is a unique closed semi-circle containing $\Lambda_\nu$ so that the end points of this semi-circle
belong to $\Lambda_\nu$. Finally, any closed $\sigma$-invariant set $\Lambda\subset {\bf S^1}$ which is contained in a semi-circle is a rotation set of $\sigma$.

\section{Accessibility}\label{access}
We define a {\it telescope} following essentially~\cite{Pr}.
Given $x\in J(f)$, $r>0$, $\delta>0$, $k\in {\bf N}$ and $\kappa\in (0,1)$, an
$(r, \kappa, \delta, k)$-telescope at $x\in J$ is collections of times
$0=n_0<n_1<...<n_k=n$ and disks $B_l=B(f^{n_l}(x), r)$, $l=0,1,...,k$ such that,
for every $l>0$:
(i) $l/n_l>\kappa$, (ii) there is a univalent branch $g_{n_l}: B(f^{n_l}(x), 2r)\to {\bf C}$ of $f^{-n_l}$
so that $g_{n_l}(f^{n_l}(x))=x$ and, for $l=1,...,k$, $\dist(f^{n_{l-1}}\circ g_{n_l}(B_l), \partial B_{l-1})>\delta$
(clearly, here $f^{n_{l-1}}\circ g_{n_l}$ is a branch of $f^{-(n_l-n_{l-1})}$ that maps $f^{n_l}(x)$ to
$f^{n_{l-1}}(x)$).
The trace of the telescope is a collection of sets $B_{l,0}=g_{n_l}(B_l)$, $l=0,1,...,k$.
We have: $B_{k,0}\subset B_{k-1,0}\subset...\subset B_{1,0}\subset B_{0,0}=B_0=B(x,r)$.
By the {\it first point of intersection} of a ray $R_t$, or an arc of $R_t$, with a set $E$ we mean a point of $R_t\cap E$ with
the minimal level (if it exists).
\begin{theorem}\label{tele}\cite{Pr}
Given $r>0$, $\kappa\in (0,1)$, $\delta>0$ and $C>0$ there exist
$M>0$,
$\tilde l,\tilde k\in \N$ and $K>1$
such that for every $(r, \kappa, \delta, k)$-telescope the following hold.
Let $k>\tilde k$. Let $u_0=u$ be any point at the boundary of $B_k$ such that $G(u)\ge C$.
Then there are indexes $1\le l_1<l_2<...<l_j=k$ such that $l_1<\tilde l$,
$l_{i+1}<K l_i$, $i=1,...,j-1$ as follows.
Let $u_k=g_{n_k}(u)\in \partial B_{k,0}$ and
let $\gamma_k$ be an infinite arc of
an external ray through $u_k$
between the pint $u_k$ and $\infty$.
Let $u_{k,k}=u_k$ and, for $l=1,...,k-1$,
let $u_{k,l}$ be the first point of intersection of
$\gamma_k$ with $\partial B_{l,0}$.
Then, for $i=1,...,j$,
$$G(u_{k,l_i})>M 2^{-n_{l_i}}.$$
\end{theorem}
Applying Theorem~\ref{tele} as in~\cite{Pr} (where the existence of a ray to $x$ is proved assuming
$(r, \kappa, \delta, \infty)$-telescope to $x$) we obtain the following statement. See also~\cite{BLy} for a direct proof of part (1).
\begin{coro}\label{accesscoro}
Let $X$ be a hyperbolic set for $f$.
(1) To every point $x\in X$ one can assign a non-empty set $A_x\subset {\bf S}^1$ such that
for every $t\in A_x$ the external ray $R_{t}$ lands at $x$.
(2) The set $\mathcal{A}=\{(x,t): x\in X, t\in A_x\}$ is closed in $\C\times {\bf S}^1$.
(3) Moreover, for each $\mu>0$ there is $C(\mu)>0$ such that for all $x\in X$ and all $t\in A_x$, the first intersection of $R_t$ with $\partial B(x,\mu)$ has the level at least $C(\mu)$.
\end{coro}
\begin{proof}
As $f:X\to X$ is expanding,
there exist $\lambda>1$, $\rho>0$ and $j_0$ such that, for every $x\in X$ and every $k>0$,
there exists a (univalent) branch $g_{k,x}: B(f^{kj_0}(x),\rho)\to {\bf C}$ of $f^{-kj_0}$ with
$g_{k,x}(f^{kj_0}(x))=x$ and $|g_{k,x}'(y)/g_{k,x}'(z)|<2$ for $y,z\in B(f^{kj_0}(x),\rho)$. Moreover,
$|g_{k,x}'(z)|<\lambda^{-k}$ for all $k>0$ and $x\in X$.
Therefore, there are $r>0$, $\delta>0$
and $\kappa=1/j_0$ such that
for any $k>1$, every point $x\in X$ admits an $(r, \kappa, \delta, k)$-telescope with $n_k=kj_0$.
In fact, $n_i=ij_0$
for $i=0,1,\cdots,k$. Let $B_{k,0}(x)\subset B_{k-1,0}(x)\subset\cdots\subset B_{1,0}(x)\subset B_{0,0}(x)$ be the corresponding trace.
Define $C_0=\inf_{y\in J(f)} \max_{z\in B(y, r)} G(z)$. It is easy to see that $C_0>0$.
For each $x\in X$ we choose a point $u(x)\in \partial B(x,r)$ such that $G(u(x))\ge C_0$.
By Theorem~\ref{tele}, there are $\tilde l$ and $\tilde k$ such that for each $k>\tilde k$ and each $x\in X$ the following hold. There are $1\le l_{1,k}(x)<l_{2,k}(x)<\cdots<l_{j^x_k,k}(x)=k$ such that
$l_{1,k}(x)<\tilde l$, $l_{i+1,k}(x)<K l_{i,k}(x)$, $i=1,\cdots,j^x_k-1$
. Let $\gamma_k(x)$ be an arc of an external ray between the point $u_k(x)=g_{k,x}(u(f^{kj_0}(x))$ and $\infty$.
Let $u_{k,l}(x)$ be the first intersection of $\gamma_k(x)$ with
$\partial B_{l,0}(x)$. Then, for $i=1,\cdots,j^x_k-1$,
\begin{equation}\label{level}
G(u_{k,l_{i,k}}(x))>M 2^{-l_{i,k}(x) j_0}.
\end{equation}
For all $i=1,\cdots,j^x_k-1$,
\begin{equation}\label{acc1}
i\le l_{i,k}(x)<K^i \tilde l.
\end{equation}
Denote by $t_k(x)$ the argument of an external ray that contains the arc $\gamma_k(x)$.
It is enough to prove in the situation above
\begin{lemma}\label{cl0}
(i) If $(x_m)_{m>0}\subset X$, $x_m\to y$ and $t_{k_m}(x_m)\to \tau$ for some $k_m\to \infty$, then $R_\tau$ lands at $y$. (ii) Moreover, for each $\mu>0$ there is $C(\mu)>0$ such that for all pairs $(y, \tau)$ like this the first intersection of $R_\tau$ with $\partial B(y,\mu)$ has the level at least $C(\mu)$.
\end{lemma}
Indeed, assuming this lemma, we can define $A_x$ as the set of all angles $t$ so that there exist $x_m\in X$ and $k_m\to \infty$ with $x_m\to x$ and $t_{k_m}(x_m)\to t$.
The set $A_x$ is not empty because one can take $x_m=x$ for all $m$ and $k_m\to\infty$ such that $\{t_{k_m}(x)\}$ converges.
By Lemma \ref{cl0}(i),
the ray $R_t$ indeed lands at the point $x$.
It is then an elementary exercise to check that the set $\mathcal{A}$ is closed. Claim (ii) implies obviously the part (3).

So, let's prove Lemma \ref{cl0}. Let $(y,\tau)$ be as in the lemma. Pick any $\mu\in (0,r)$. Fix $l_0$ so that $\lambda^{-l_0} r<\mu/2$ and let
\begin{equation}\label{acc2}
C(\mu)=\frac{M}{2^{\tilde l j_0 K^{l_0}}}.
\end{equation}
There is $m_0$ such that for all $m>m_0$ and all $l>l_0$,
$B_{l_0,0}(x_m)\subset B(y,\mu)$. Then, by (\ref{level})-(\ref{acc1}), for every $m>m_0$,
\begin{equation}\label{firstint}
G(u_{k_m, l_{l_0,k_m}}(x_m))>C(\mu).
\end{equation}
Hence, for every $m>m_0$ the first intersection of $\gamma_{k_m}(x_m)$ with $B(y,\mu)$ has level at least $C(\mu)$.
This implies that, given $C\in (0, C(\mu))$, for any $m>m_0$, an arc
of the ray $R_{t_{k_m}(x_m)}$ between the levels $C$ and $C(\mu)$ does not exit $B(y,\mu)$. As this sequence of arcs tend uniformly, as $m\to\infty$, to an arc of $R_\tau$ between the levels $C$ and $C(\mu)$,
this latter arc is contained in $B(y,\mu)$. As $C>0$ can be chosen arbitrary small, the ray $R_\tau$ must land at $y$ and its first intersection
with $B(y,\mu)$ has the level at least $C(\mu)$. Lemma \ref{cl0} is proved.
\end{proof}
\section{A combinatorial property}
Let $f$ be an infinitely-renormalizable quadratic polynomial.
First, we prove the following combinatorial fact (a maximality property) about $f$.

In the course of the further proofs the following well-known easy statement about expanding maps is used
(for a more general theorem about expansive maps, see e.g. \cite{expansive}):
\begin{prop}\label{exp}
Let $h$ be a homeomorphism of a compact metric space $(S,d)$ onto itself which is expanding in the following sense:
there are $\delta>0$ and $\lambda>1$ such that $d(h(x), h(x'))\ge\lambda d(x, x')$ whenever $d(x,x')<\delta$.
Then $S$ is a finite set.
\end{prop}
Let $\omega(t)$ be the omega-limit set of $t\in {\bf S}^1$ under $\sigma: t\mapsto 2t (mod 1)$ and $\omega(E)=\cup_{t\in E} \omega(t)$.
\begin{lemma}\label{2p}

(i) $\sigma^{-1}(\tilde K_c)\subset \omega(\tilde K_c)$

(ii) $\tilde J_\infty=\omega(\tilde K_c)$.
\end{lemma}
\begin{remark}\label{ii}
Only part (i) is used in the proof of Theorem \ref{thm:exp}. Part (ii) seems to be of an independent interest and is included for completeness.
\end{remark}
\begin{proof}[Proof of (i)]
(a) Consider $\sigma: \tilde J_\infty\to \tilde J_\infty$. Each $t\in J_\infty$ belongs to $\tilde K$, for one and only one component $K$ of $J_\infty$.
By (B), one and only one component $K^{-1}$ of $f^{-1}(K)$ is a component of $J_\infty$, moreover, $f: K^{-1}\to K$ is a homeomorphism if $K\neq K_c$ and $f: K_0\setminus \{0\}\to K_c\setminus \{c\}$ is two-to-one.
Hence, $\sigma^{-1}(t)\cap \tilde J_\infty$ is a single point for $t\notin\tilde K_c$ and two different points for $t\in\tilde K_c$.
On the other hand, for each $t\in {\bf S}^1$, $\sigma$ maps $\omega(t)$ onto itself. Since $\sigma: \tilde J_\infty\to \tilde J_\infty$ has no periodic points, by Proposition \ref{exp},
for each $t\in \tilde J_\infty$, the expanding map $\sigma: \omega(t)\to \omega(t)$ is not injective. Therefore, $\omega(t)\cap \tilde K_c\not=\emptyset$ for all $t\in\tilde J_\infty$.

(b) If $\tilde K_c$ consists of a single angle $\tau_c$, then (a) implies that there is a single point $t\in\omega(\tau_c)$ such that $\sigma^{-1}(t)\subset\omega(\tau_c)$
and, moreover, $t=\tau_c$.
Therefore,
$\sigma^{-1}(\tau_c)\subset \omega(\tau_c)$
and we are done in this case.

(c) It remains to deal with a two-point set $\tilde K_c=\{\tau_1, \tau_2\}$.
Let us assume the contrary, i.e., $\sigma^{-1}(\tau_1)\cup \sigma^{-1}(\tau_2)$ is not a subset of
$\omega(\tilde K_c)=\omega(\tau_1)\cup \omega(\tau_2)$.
Hence, by (a) and by the assumption, either
$\sigma^{-1}(\tau_1)$ or $\sigma^{-1}(\tau_2)$ is a subset of each
$\omega(\tau_i)$, $i=1,2$.
Let, say, $\sigma^{-1}(\tau_1)\subset \omega(\tau_1)\cap\omega(\tau_2)$.
By the assumption, there is $\tau\in \sigma^{-1}(\tau_2)$ such that $\tau\notin \omega(\{\tau_1,\tau_2\})$.
Let $L$ be a (open) semi-circles ${\bf S^1}\setminus \sigma^{-1}(\tau_1)$ that contains $\tau$.
We claim that it is enough to show that for each $p_n$ and some $j_n>0$ the arc $L$ contains
$\sigma^{j_n p_n-1}(\tau_1)$. Indeed, assume this is the case.
Then, by (D1)$3_\infty$, Sect. \ref{prel}, $L$ must contain
one of the arcs ${\bf S^1}\setminus \{\sigma^{j_n p_n-1}(\tau_1), \sigma^{j_n p_n-1}(\tau_2)\}$ and, by (D1)$2_\infty$,
the lengths of all such arcs are uniformly away from $0$. Hence, there is a subsequence $n_i\to \infty$
such that the sequences $\sigma^{j_{n_i} p_{n_i}-1}(\tau_1)$ and $\sigma^{j_{n_i} p_{n_i}-1}(\tau_2)$ converge
to points $a_1$ and $a_2$ respectively, where $a_1\not=a_2$ and $a_1, a_2\in \overline{L}$.
On the other hand, $a_1,a_2\in\tilde K_0$. But, from the assumption, $\tilde K_0\cap L=\{\tau\}$.
Therefore $\{a_1, a_2\}\subset\{\tau, \sigma^{-1}(\tau_1)\}$. Since $a_1,a_2\in\omega(\{\tau_1,\tau_2\})$ while $\tau\notin \omega(\{\tau_1,\tau_2\})$,
we conclude that $\{a_1, a_2\}=\sigma^{-1}(\tau_1)$. But then $\sigma^{j_{n_i} p_{n_i}}(\tau_1)$ and $\sigma^{j_{n_i} p_{n_i}}(\tau_2)$
converge to the same point $\tau_1$ which is a contradiction with (D1)$2_\infty$.

(d) To show that, for each $n$, the arc $L$ contains a
point $\sigma^{j_n p_n-1}(\tau_1)$, for some $j_n>0$, let us assume the contrary. So, we fix $n>0$ and assume that
$\Lambda:=\{\sigma^{j p_n-1}(\tau_1): j>0\}\subset L':={\bf S^1}\setminus L$.
The idea is to show that the set $\Lambda$ corresponds to a rotation set $\Lambda_0$ of $\sigma: {\bf S^1}\to {\bf S^1}$ and use a certain structure of later sets \cite{BS},
see (D2) of Sect. \ref{prel}, to arrive at a contradiction.

Following the notations of (C), let $\Omega_n^0=f^{-1}(\Omega_n)$ and $U_n^0=f^{-1}(U_n)$.
To connect it to other notations introduced in (C), $\Omega_n^0=U_{n,p_n}=f^{p_n-1}(U_n)$ and $U_n^0=U^1_{n,p_n}$.
Then $J_n=\{z| f^{j p_n}(z)\in \overline{U_n^0},  G(f^{j p_n}(z))\le 10, j=0,1,\cdots\}$.
The domain $\Omega_n^0$ is bounded by two bi-infinite curves $\Gamma_n$, $\Gamma_n'$ (components of $f^{-1}(\partial\Omega_n$)) and two angular (closed) "arcs at infinity"  which are components $S_{n,0}, S_{n,0}'$ of $\sigma^{-1}(S_{n,1})$ where $S_{n,1}=[t_n, \tilde t_n]$, $0<t_n<\tilde t_n<1$. In fact, $S_{n,0}, S_{n,0}'$ are two components of $s_{n,p_n}=\sigma^{p_n-1}(s_{n,1})$.
Arguments of rays entering $\overline{U_n^0}$ form the set $s^1_{n,p_n}\subset s_{n,p_n}$ consisting of four $1$-'widows' such that
$\sigma^{p_n}$ is a homeomorphism of each of them onto either $S_{n,0}$ or $S_{m,0}'$.
By (\ref{tildejany}), $\widetilde{J_n}=\widetilde{f^{p_n}(J_n)}=\{t| \sigma^{k p_n}(t)\in s^1_{n,p_n}, k=0,1,...\}$.

Let us specify $\Gamma_n$ to be such component of $f^{-1}(\partial\Omega_n)$ that contains $f^{p_n-1}(\beta_n)$  (in other words, $f^{p_n}(\Gamma_n)=\Gamma_n$),
and $S_{n,0}$ to be the first "arc" as one goes from $\Gamma_n$ to $\Gamma_n'$ inside of $\Omega_n^0$ in the counterclockwise direction along the "circle at infinity".
In other words, $S_{n,0}$ denotes the component of $\sigma^{-1}(S_{n,1})$ that contains $\sigma^{p_n-1}(t_n)$ in its boundary.
Now, let $\epsilon(t)=0$ if $t\in S_{n,0}$ and $\epsilon(t)=1$ if $t\in S_{n,0}'$.
To every $t\in\widetilde{J_n}$
we assign a point $\theta(t)\in {\bf S}^1$ as follows:
$$\theta(t)=\sum_{j=0}^\infty \frac{\epsilon(\sigma^{j p_n}(t))}{2^{j+1}}.$$
Then $\theta(\widetilde{J_n})={\bf S^1}$, $\theta\circ \sigma^{p_n}=\sigma\circ \theta$,
and $\theta:\widetilde{J_n}\to {\bf S^1}$ extends to a continuous monotone degree one map ${\bf S^1}\to {\bf S^1}$, see \cite{dangle}, \cite{Mi1}, \cite[Theorem 3]{Leray}.
Moreover, $\theta$ is injective on a subset $T=\{t\in \widetilde{J_n}| \sigma^{k p_n}(t)\notin \partial(S_{n,0}\cup S_{n,0}'), k\ge 0\}$. Note that $\widetilde{J_\infty}\subset T$ where
$\widetilde{J_\infty}$ is closed.
Besides, $\tilde K_0=\sigma^{-1}(\{\tau_1,\tau_2\})$ and $\Lambda=\{\sigma^{j p_n-1}(\tau_1): j>0\}$ are subsets of $\widetilde{J_\infty}\subset T$.
If $\theta_0:=\theta(\tau_1/2)$ then $\theta_1:=\theta(\tau_1/2+1/2)=\theta_0+1/2$. Therefore,
from the assumption, the set
$\Lambda_0:=\theta(\Lambda)$ is a subset of a semi-circle $C_{\theta_0}$ with end points $\theta_0$ and $\theta_1$.
As $\sigma^{p_n}(\Lambda)\subset\Lambda$, then $\sigma(\Lambda_0)\subset \Lambda_0$.
It follows \cite{BS} (see (D2) of Sect. \ref{prel} in the present paper), that the set $\overline{\Lambda_0}$ is a rotation set of $\sigma:{\bf S^1}\to {\bf S^1}$.
Let $E\subset \overline{\Lambda_0}$ be a closed subset such that $\sigma: E\to E$ is minimal. As $\overline\Lambda\subset\widetilde{J_\infty}\subset T$
where $\widetilde{J_\infty}$ contains no periodic points of $\sigma^{p_n}$, $E$ contains no periodic points of $\sigma$.
Hence, $E$ is infinite. By \cite{BS}, see (D2), for every closed infinite minimal rotation set of $\sigma$ there is a unique closed semi-circle containing it and in this case the end points of the semi-circle belong to the set.
Thus $\theta_0, \theta_0+1/2\in E\subset\overline{\Lambda_0}\subset\overline{C_{\theta_0}}$.
Coming back to the $f$-plane, we find two sequences $j_i, j_i'\to \infty$ so that $\sigma^{j_i p_n-1}(\tau_1)\to \tau_1/2$ and $\sigma^{j_i' p_n-1}(\tau_1)\to \tau_1/2+1/2$ inside of the same semi-circle $L'$
with the end points $\tau_1/2$, $\tau_1/2+1/2$. But then  $\sigma^{j_i p_n}(\tau_1)$ and $\sigma^{j_i' p_n}(\tau_1)$ both tend to $\tau_1$ from different sides, in a contradiction with (D1)$1_\infty$.
This completes the proof of part (i) of the statement.

Let us prove part (ii). Obviously, $\omega(\tilde K_c)\subset \tilde J_\infty$. To show the opposite,
given $t\in\tilde J_\infty$ let $K$ be a unique component of $J_\infty$ such that $t\in\tilde K$.
Let $K=\cap_{n\ge 1} f^{j_n}(J_n)$.
By Lemma \ref{tilde}, there is an alternative: either (1) $p_n-j_n\to\infty$ as $n\to\infty$ and
$\tilde K=\cap_{n=1}^\infty s_{n,j_n}$ or (2) there is $N\ge 0$ such that $f^N(K)=K_0$, i.e., $p_n-j_n=N$ for all $n$.
In case (2), $f^N: K\to K_0$ is a homeomorphism, which implies that $\sigma^N:\tilde K\to \tilde K_0$ is also a homeomorphism.
Hence, if $s\in\tilde K_0$ is a point of $\omega(\tilde K_c)$, then $\sigma^N|_{\tilde K}^{-1}(s)$ is also a point of $\omega(\tilde K_c)$.
On the other hand, by part (i), $\tilde K_0=\sigma^{-1}(\tilde K_c)\subset \omega(\tilde K_c)$. Therefore, $t\in\omega(\tilde K_c)$
in case (2). In case (1), $\{t\}=\cap_{n\ge 1}s^{(t)}_{n,j_n}$ where $s^{(t)}_{n,j_n}$ is one of the two 'windows' of $s_{n,j_n}$,
and $\sigma^{j_n}: s^{(t)}_{n,1}\to s^{(t)}_{n,j_n}$ is a homeomorphism where of course $s^{(t)}_{n,1}$ is one of the two
'windows' of $s_{n,1}$. Hence, it's enough to show that each 'window' of $s_{n,1}$ contains points of the orbit of $\tilde K_c$.
As $\tilde K_0=\cap_n s_{n,1}$, in the case $\# \tilde K_c=2$, for all $n$ big enough, each 'window' of $s_{n,1}$ contains one
and only one of the two points of $\tilde K_c$, and we are done in this case.
In the remaining case, $\tilde K_c$ is a single point which lies in one of the two 'widows' of $s_{n,1}$, for each $n$.
But, as $\sigma^{p_n}$ maps each 'window' of $s_{n,1}$ onto $S_{n,1}\supset s_{n,1}$ homeomorphically it is obvious that each 'window' of $s_{n,1}$ contains infinitely many points of the
$\sigma^{p_n}$-orbit of any $t_0\in \widetilde{f(J_n)}$ provided no point of this orbit hits
$\partial S_{n,1}$. In particular, this applies to points of $\tilde K$
This completes the proof.
\end{proof}
\begin{remark}
In case (1) we proved more: if $t\in\tilde J_\infty$ is such that $\sigma^N(t)\notin \tilde K_0$ for all $N\ge 0$, then $t\in\omega(\tau)$
for each $\tau\in\tilde K_c$.
\end{remark}

\section{Proof of Theorem~\ref{thm:exp}}\label{thm1}
1. Assume the contrary and let $X\subset J_\infty$ be a compact $f$-invariant hyperbolic set. In particular, Corollary~\ref{accesscoro} applies.

2. Replacing $X$ by its subset if necessary we can assume that
$f:X\to X$ is a minimal map.

3. $0\notin X$, hence, $c\notin X$, too.

4. As $J_\infty$ contains no cycles, $X$ is an infinite set.
If we assume that $f:X\to X$ is one-to-one, then
$f:X\to X$ is an expanding homeomorphism of a compact set, therefore,
$X$ is finite, a contradiction with Proposition \ref{exp}.

5. Thus $f:X\to X$ is not one-to-one.
Then, by (B), $X_c:=X\cap K_c\not=\emptyset$. On the other hand, by step 3, $c\notin X_c$.




6. By (C), $\tilde K_c$ consists of either one or two arguments. As any point of $X$ is accessible, $X_c$ consists of either one or two different points.
Let $x_1\in X_c$ and $x_2\in J(f)$ be any other point.
Given $r>0$ and $n$, let $W_n(x_1, r)$ be a component of $B(x_1,r)\cap \Omega_n$ (see (C), Sect. \ref{prel}, where $\Omega_n$ is defined) that contains the point $x_1$.
We use repeatedly the following

{\bf Claim 1}. {\it Given $\hat r>0$ and $\hat C>0$, there is $\hat n\in {\bf N}$ as follows.
For $k=1,2$, suppose that, for some angles $\hat t_k$,  the ray $R_{\hat t_k}$ lands at $x_k$ and let $z_k$ be the first intersection of $R_{\hat t_k}$ with $\partial B(x_k,\hat r/2)$.
Assume:
(a) $G(z_k)>\hat C$ for $k=1,2$,
(b) $x_2\in\Omega_n$ and $|x_1-x_2|<\hat r/3$,
(c) one of the following holds:
(i) $t_n-\tilde t_n\to 0$ as $n\to\infty$, or
(ii) $\hat t_1,\hat t_2$ belong to a single component of
$s_{n,1}=[t_n, t_n']\cup [\tilde t_n',\tilde t_n]$.
Then $x_2\in W_n(x_1,\hat r)$
for each $n>\hat n$}.

Indeed, the length of each component of $s_{n,1}$ goes to zero as $n\to\infty$.
Hence, as $\hat r$ and $\hat C$ are constants and $n$ is big enough, condition (c) implies that a curve which consists of an arc of $R_{\hat t_1}$ from $x_1$ to $z_1$, then the shortest arc of the equipotential
containing $z_1$ from $z_1$ to the first
intersection $u_2$ with $R_{\hat t_2}$ and then back along $R_{\hat t_2}$ from $u_2$ to $x_2$ belongs to
$\Omega_n$ and $B(x_1,\hat r)$ simultaneously.

7. Fix $r>0$ small enough. Let $a\in X_c$ and $a_{-1}\in X$ be such that $f(a_{-1})=a$.
As $a_{-1}\in K_0\cap X$, there is its uniquely defined backward orbit
$\{a_{-m}\}_{m=1}^\infty\subset X$, $f(a_{-m-1})=a_{-m}$, $m\ge 1$. Let $a'$ be a limit point of the sequences $a_{-p_n}$, i.e, $a'=\lim_{i\to\infty}a_{-p_{n_i}}$.
As $a_{-p_n}\in f(J_n)$,
$a'$ belongs to $K_c$ and $X$ at the same time, that is, $a'\in X_c$.

{\bf Claim 2}. {\it For all $i$ large enough, $a_{-p_{n_i}}\in W_{n_i}(a',r)$}.

Indeed,
by Corollary~\ref{accesscoro} there is a subsequence $(n'_i)$ of $(n_i)$ and a converging sequence $t_i\in A_{a_{-(p_{n'_i})}}$ such that $t:=\lim_{i\to\infty} t_i$ and $t\in A_{a'}$. We have: $t_i\in s_{n'_i}$ for all $i$.
Then Claim 1 of Step 6 applies
for each $i$ big enough, with $\hat r=r$, $\hat C=C(r/2)$, $x_1=a'$, $x_2=a_{-p_{n'_i}}$, $\hat t_1=t$, $\hat t_2=t_i$ and $z_1,z_2$ defined by this data as in Claim 1. Indeed, (a) holds by Corollary~\ref{accesscoro}(3) and (b) is obvious (note that $a_{-p_n}\in f(J_n)\setminus\{\beta_n\}\subset\Omega_n$). Moreover, if (ci) breaks down, since $t_i\to t$, then $t_i$ and $t$ must lie at the same component of $s_{n'_i}$.

By the conclusion of Claim 1, $a_{-p_{n'_i}}\in W_{n'_i}(a',r)$ for each $i$ big enough. Finally, as $A_{a'}$ consists of at most four points (therefore, the sequence $(a_{-n_i})$ has at most four limit points), Claim 2 follows.

8. Consider the case $X_c=\{a\}$. Let $f^{-1}(a)=\{a_{-1}, a^*_{-1}\}$. By steps 2, 3 and 5, $a_{-1}\neq a^*_{-1}$ and $a_{-1},a^*_{-1}\in X$. As $X_c=\{a\}$, there is a subsequence $(n_i)$ such that
backward images $a_{-p_{n_i}}$, $a^*_{-p_{n_i}}$ of $a_{-1}$, $a^*_{-1}$ respectively tend to the same point $a$.
By Claim 2 Step 7, for each $i$ large, $a_{-p_{n_i}}$, $a^*_{-p_{n_i}}\in W_{n_i}(a,r)$.
Therefore, the following two sets (which are preimages of $W_{n_i}(a,r)$ by $f^{p_{n_i}}$):
$V_{n_i}:=g_{p_{n_i},a_{-(p_{n_i})}}(W_{n_i}(a, r))$ and $V_{n_i}^*:=g_{p_{n_i},a^*_{-(p_{n_i})}}(W_{n_i}(a, r))$, are
disjoint with their closures (because preimages of $B(a,r)$ along points of $X$ shrink exponentially) and both are
contained in $W_{n_i}(a,r)$.
Fix such $n=n_i$. Denote for brevity $J_{c,n}=f(J_n)$. Then we get that, for every $j>0$, $2^j$ preimages
of $a\in J_{c,n}$ by the map $f^{j p_n}: J_{c,n}\to J_{c,n}$ are contained in
the (disjoint) closures of $V_n$ and $V_n^*$. As the set of all those preimages are dense in $J_{c,n}$,
we get a contradiction with the fact that $J_{c,n}$ is a continuum.

9. Consider the remaining case $X_c=\{a, b\}$, $a\not=b$. As $\# \tilde K_c\le 2$, each point $a$ and $b$ is accessible by a single ray $R_{t(a)}$ and $R_{t(b)}$ respectively.
Hence, any point $u$ from the {\it grand orbits} of $a$ and $b$ is a landing point of precisely one ray $R_{t_u}$.
Let us prove that $f^{-1}(X_c)\subset X$. Let $f(w)=x\in X_c$. By Lemma~\ref{2p}(i), one can find $y\in X_c$ and $m_i\to \infty$ such that $\sigma^{m_i}(t_y)\to t_w$
and $f^{m_i}(y)$ tends to some point $\tilde w\in X$. By Corollary~\ref{accesscoro}, $t_w\in A_{\tilde w}$. But $\sigma(t_w)=t_x$, hence, $f(\tilde w)=x$ and $A_{\tilde w}=\{t_w\}$.
Thus $w=\tilde w\in X$.

10. We have just proved that $\{a_{-1}, a^*_{-1}\}=f^{-1}(a)\subset X$ and $\{b_{-1}, b^*_{-1}\}=f^{-1}(b)\subset X$.
Now, we repeat the consideration as in Step 8. The sequences $a_{-(p_n)}$, $a^*_{-(p_n)}$, $b_{-(p_n)}$, $b^*_{-(p_n)}$ must have all their limit points in $X_c$.
As $r>0$ is small enough, $\overline{B(a,r)}\cap \overline{B(b,r)}=\emptyset$.
By Claim 2 of Step 7, for each $n$ large, all $4$ points $a_{-(p_n)}, a^*_{-(p_n)}, b_{-(p_n)}, b^*_{-(p_n)}$ are in $W_n(a, r)\cup W_n(b, r)$.
Fixing $n$ large, for each disk $B(x, r)$, $x\in \{a,b\}$, there are two univalent branches of $f^{-p_n}$ which are defined in $B(x, r)$
such that they map $W_n(x, r)$
inside $W_n(a, r)\cup W_n(b, r)$. Hence, for every $j>0$, $4^j$ preimages
of $X_c\in J_{c,n}$ by the map $f^{j 2p_n}: J_{c,n}\to J_{c,n}$ are contained in
the (disjoint) closures of $W_n(a,r)$ and $W_n(b,r)$. As the set of all those preimages are dense in $J_{c,n}$,
we again get a contradiction with the fact that $J_{c,n}$ is a continuum.
\begin{remark}\label{high}
The combinatorial property for quadratic polynomials of Lemma \ref{2p}
implies that if $X\subset J_\infty$ is a minimal hyperbolic set then $f^{-1}(X)\cap J_\infty=X$ provided $f$ is quadratic, and this leads to a contradiction. Therefore, a small modification of the proof gives us
the following claim for infinitely-renormalizable unicritical polynomial $f(z)=z^d+c$ with any $d\ge 2$:
$J_\infty$ contains no hyperbolic sets $X$ such that $f:X\to X$ is minimal and $f^{-1}(X)\cap J_\infty=X$.
\end{remark}

\section{Final remarks}
A hyperbolic set of a rational map always carries an invariant measure with a positive Lyapunov exponent.
Conjecturally, $J_\infty$ as in Theorem \ref{thm:exp} never carries such a measure.
We cannot prove this conjecture in the full generality so far,
but we can easily prove at least that
  $F:=f|_{J_\infty}$ is not "chaotic". Namely,

1. Every $F$-invariant probability measure $\mu$ has zero entropy, $h_\mu(F)=0$.

2. Topological entropy of $F$ is zero, $h_{\rm top}(F)=0$.

Proving it we can assume $\mu$ is ergodic due to Ergodic Decomposition Theorem, see e.g. \cite[Theorem 2.8.11a)]{PU}.
Start by observing that, for every $n$ and $0\le j<p_n$, $\mu(f^j(J_n))=1/p_n$, hence,
$\mu$ has no atoms and $\mu(K)=0$ for every component $K$ of $J_\infty$. Therefore, if $J'_\infty=J_\infty\setminus\cup_{n\in\Z}f^n(K_0)$
where $K_0$ is the component of $J_\infty$ containing $0$ (see (B) of Sect. \ref{prel}), then $F: J'_\infty\to J'_\infty$ is an automorphism.
Suppose to the contrary that $h_\mu(F)>0$. Then 1. follows from Rokhlin entropy formula, \cite[Theorem 2.9.7]{PU}, saying that $h_\mu(F)=\int\log {\rm Jac}_\mu (F) d\mu$. Here
${\rm Jac}_\mu$ is Jacobian with respect to $\mu$, equal to 1 $\mu$-a.e,. since $\mu$ must be supported on $J'_\infty\subset J_\infty$ where $F$ is an automorphism. A condition to be verified  to apply Rokhlin formula is the existence of a one-sided countable generator of bounded entropy, proved to exist by Ma\~n\'e, see e.g. \cite[Lemma 11.3.2]{PU}
and inclusion \cite[(11.4.8)]{PU}, due to positive Lyapunov exponent $\chi_\mu(F):=\int\log|F'|d\mu\ge {1\over 2} h_\mu(F)>0$ (Ruelle's inequality). Thus $h_\mu(F)>0$ has led to a contradiction.

2. follows from 1. by variational principle $h_{\rm top}(F)=\sup_\mu h_\mu(F)$.

Compare item 4 in Section \ref{thm1}.
Here finiteness of $X$ is replaced by zero entropy.

The same proof with obvious modifications holds for $f(z)=z^d+c$, $d\ge 2$.


\end{document}